\documentclass[a4paper, 11pt]{article}
\usepackage[margin=1in]{geometry}
\usepackage[T1]{fontenc}
\usepackage[utf8]{inputenc}
\usepackage{lmodern}
\usepackage{amsmath,amsthm,amssymb,mathtools}
\usepackage{microtype}
\usepackage[hidelinks]{hyperref}
\usepackage[nameinlink,noabbrev]{cleveref}
\title{On a Question of Gowers on Clique Differences}
\author{Ryan Alweiss\thanks{Department of Pure Mathematics and Mathematical Statistics and Trinity College, University of Cambridge. Email: ra699@cam.ac.uk. Research supported by an NSF Mathematical Sciences Postdoctoral Fellowship.}}
\date{September 3, 2025}
\begin{document}
\maketitle

\begin{abstract}
We solve a question of Gowers from 2009 on clique differences in chains, thus ruling out any Sperner-type proof of the polynomial density Hales–Jewett Theorem for alphabets of size 2.
\end{abstract}

\theoremstyle{definition}
\newtheorem{theorem}{Theorem}[section]
\newtheorem{lemma}[theorem]{Lemma}
\newtheorem{proposition}[theorem]{Proposition}
\theoremstyle{definition}
\newtheorem{definition}[theorem]{Definition}
\newtheorem{conjecture}[theorem]{Conjecture}
\newtheorem{question}[theorem]{Question}
\theoremstyle{remark}
\newtheorem{remark}[theorem]{Remark}

\section{Introduction}

The polynomial density Hales–Jewett question is one of the central remaining open questions in ergodic Ramsey theory, and is a common generalization of the Bergelson–Leibman theorem~\cite{BL} and the density Hales–Jewett theorem~\cite{DHJ}.   One special case was proposed by Gowers as a potential polymath project~\cite{GowersBlog}.

\begin{conjecture}[Gowers 2009, \cite{GowersBlog}] For every $\delta>0$, there exists $n$ such that if $A$ is any collection of at least $\delta 2^{\binom{n}{2} }$ graphs with vertex set $\{1, \cdots, n\}$, then $A$ contains two distinct graphs $G$ and $H$ so that $G \subset H$ and $H \setminus G$ is a clique.
\end{conjecture}

In the same blog post, Gowers asks the following question in an attempt to prove Conjecture 1.1.


\begin{question}[Gowers, 2009]
Is it true that for every $\delta>0$ there exists a sequence of distinct graphs $G_1 \subset G_2 \subset \cdots \subset G_r$ such that for every subset $A \subset \{1,2,\dots,r\}$ of size at least $\delta r$ there exists $i<j$ such that $i,j \in A$ and $G_j \setminus G_i$ is a clique?
\end{question}

If the answer to the above question is positive, then it seems plausible that, by uniformly covering the set of all graphs by sequences of the above form, Conjecture~1.1 would follow from a positive answer to Question~1.2; see~\cite{GowersBlog} for more details.

Given a sequence of distinct graphs $G_1 \subset \cdots \subset G_r$, we form the graph $G$ on them by connecting $G_i, G_j$ for $i<j$ when $G_j \setminus G_i$ is a clique.  In this note we settle Question~1.2 in the negative, and we show that $G$ always has an independent set of size at least $\frac{r}{20}$.

\begin{theorem}\label{thm:main}
For every sequence of distinct graphs $G_1 \subset G_2 \subset \cdots \subset G_r$, there exists $A\subset \{1,2,\dots,r\}$ of size at least $r/20$ so that for all $i,j \in A$, $G_j \setminus G_i$ is not a clique.
\end{theorem}

\section{Proof of Theorem 1.3}

\begin{lemma}\label{lem:chain}
Given $a<b<c<d$, if $G_c\setminus G_a$ and $G_d\setminus G_b$ are cliques, then $G_c\setminus G_b$ is also a clique.
\end{lemma}

\begin{proof}
$G_c\setminus G_a$ and $G_d\setminus G_b$ are cliques, so their nonempty edge intersection $G_c\setminus G_b$ is also a clique.
\end{proof}

\begin{lemma}\label{lem:triple}
For any $y$, it cannot be true that all of $G_y, G_{y+1}, G_{y+2}$ have at least three right-neighbors and three left-neighbors in $G$.
\end{lemma}

\begin{proof}
Assume this was the case. Then there is some $z \ge y+3$ so that $G_z \setminus G_y$ is a clique. Also there is some $x < y$ so that $G_{y+1} \setminus G_x$ is a clique. By Lemma~\ref{lem:chain} on $(x, y, y+1, z)$, then $G_{y+1} \setminus G_y$ is a clique. Similarly, $G_{y+2} \setminus G_{y+1}$ is a clique. Also there is some $u < y$ so that $G_{y+2} \setminus G_u$ is a clique, and then by Lemma~\ref{lem:chain} on $(u, y, y+2, z)$, $G_{y+2} \setminus G_y$ is a clique, but it is the edge-disjoint union of the cliques $G_{y+2} \setminus G_{y+1}$ and $G_{y+1} \setminus G_y$, which is a contradiction.
\end{proof}

Call a graph $G_i$ \emph{good} if it has either at most two right-neighbors in $G$ or at most two left-neighbors in $G$. For any $G_1 \subset G_2 \subset \cdots \subset G_r$, at least $\frac{r-2}{3}$ are good by Lemma~\ref{lem:triple} because we cannot have three bad in a row. So without loss of generality at least $\frac{r-2}{6}$ of the graphs have at most two right-neighbors. This means we can go through these graphs from left to right and greedily select an independent set of size at least $\frac{r-2}{18} \ge \frac{r}{20}$. We are done.

Following the initial proof, Noga Alon communicated an improvement of the constant in Theorem 1.3.

\begin{lemma}
There are no $G_{3i-2}, G_{3i-1}, G_{3i}$ so that:
\begin{itemize}
  \item There is some $j > 3i - 1$ so that $G_j \setminus G_{3i-2}$ is a clique.
  \item There is some $k > 3i - 1$ and some $\ell < 3i - 1$ so that $G_k \setminus G_{3i-1}$ and $G_{3i-1} \setminus G_\ell$ are cliques.
  \item There is some $m < 3i - 1$ so that $G_{3i} \setminus G_m$ is a clique.
\end{itemize}

\end{lemma}

\begin{proof}
Assume this were the case. Then if $\ell = 3i - 2$, $G_{3i-1} \setminus G_{3i-2}$ is a clique. Else if $\ell < 3i - 2$, then by Lemma~\ref{lem:chain} on $(\ell, 3i-2, 3i-1, j)$, $G_{3i-1} \setminus G_{3i-2}$ is a clique anyway. Similarly, $G_{3i} \setminus G_{3i-1}$ is also a clique. Finally if $j = 3i$ or $m = 3i-2$ then $G_{3i} \setminus G_{3i-2}$ is a clique. Else if $j > 3i$ and $m < 3i-2$ then by Lemma~\ref{lem:chain} on $(m, 3i-2, 3i, j)$, $G_{3i} \setminus G_{3i-2}$ is a clique anyway. This contradicts Lemma~\ref{lem:triple}.
\end{proof}

For each $i \ge 1$, either $G_{3i-2}$ violates the first bulleted condition, $G_{3i-1}$ the second, or $G_{3i}$ the third. Select one of them correspondingly. Direct all edges from left to right. Now on this induced subgraph, every vertex $G_{3i-2}$ has outdegree $0$, every vertex $G_{3i-1}$ has either indegree $0$ (if there does not exist $\ell$) or outdegree $0$ (if there does not exist $k$), and every vertex $G_{3i}$ has indegree $0$. So every vertex has indegree or outdegree $0$, and thus there is an independent set of size at least $\lceil \lfloor r/3 \rfloor / 2 \rceil \ge \frac{r-2}{6}$.

\section{Acknowledgements}

Thanks to Noga Alon, Timothy Gowers, and an anonymous reviewer for helpful comments.

\end{document}